\documentclass[12pt, oneside]{article}   	
  \usepackage{geometry}                		
  \usepackage{amsmath}
  \usepackage{amsthm}
  \usepackage{graphicx}				
  \usepackage{amssymb}
  \usepackage{authblk}
  
  \newtheorem{theo}{Theorem}
  \newtheorem{lemm}{Lemma}
  
  \newtheorem{remark}{Remark}[section]

  \newtheorem{assumption}{Assumption}
  \newtheorem{notation}{Notation}
  \def\hoge<#1>{\langle #1 \rangle}
  \providecommand{\keywords}[1]
{
  \small
  \textbf{\textit{Keywords---}} #1
}

\DeclareMathOperator*{\argmax}{arg\,max}
\title{Parameter estimation of stochastic differential equation driven by small fractional noise}
\author{Shohei Nakajima\thanks{email:~st08m26@akane.waseda.jp}}
\author{Yasutaka Shimizu\thanks{email:~shimizu@waseda.jp}}
\affil{Department of Applied Mathmatics, Waseda University\\ 3-4-1, Shinjuku, Okubo, Tokyo, 169-8555, Japan}
\date{}
\begin{document}
\maketitle
\begin{abstract}
  We study the problem of parametric estimation for continuously observed stochastic processes driven by additive small fractional Brownian motion with Hurst index $H\in(0,1)/\{\frac{1}{2}\}$. Under some assumptions on the drift coefficient, we obtain the asymptotic normality and moment convergence of maximum likelihood estimator of the drift parameter when a small dispersion coefficient $\varepsilon\rightarrow 0$.
\end{abstract}
 \section{Introduction}
 Let $\{X_t^\varepsilon\}_{t\in[0,T]}$ be a solution to the following stochastic differential eqution:
 \begin{equation}
   \label{SDE}
   X^\varepsilon_t=x+\int_0^tb(X_s^\varepsilon,\theta_0)ds+\varepsilon B_t^H,~~~t\in(0,T],
 \end{equation}
 where $\{B_t^H\}_{t\in[0,T]}$ is a fractional Brownian motion with Hurst index $H\in(0,1)/\{\frac{1}{2}\}$ and $\theta_0\in\Theta$ is the parameter which is contained in a bounded and open convex  subset  $\Theta\subset\mathbb{R}^d$ admitting Sobolev's inequalities for embedding $W^{1,p}(\Theta)\hookrightarrow C(\bar{\Theta})$. Without loss of generality, we assume that $\varepsilon\in(0,1]$. The main purpose of this paper is the estimation of parameter $\theta\in\Theta$ from a realization $\{X_t^\varepsilon\}_{t\in[0,T]}$ when $\varepsilon\rightarrow 0.$
  In the case where $H=1/2$, that is, $\{B_t^H\}_{t\in[0,T]}$ is a Brownian motion, estimation problems have been studied by many authors. In particular, the maximum likelihood estimator (MLE) via the likelihood function based on the Girsanov density is the one of the optimal methods for estimation (see Pracasa Rao \cite{Rao}, Liptser and Shiryaev \cite{Liptser} and Kutoyants \cite{Kutoyants}).

  The parametric inference for stochastic differential equation driven by fractional Brownian motion have been studied by Brouste and Kleptsyna \cite{Brouste}, Kleptsyna and Le Breton \cite{Kleptsyna} and Tudor and Viens \cite{Tudor} when the MLE has an explicit expression. Recently, in the case when the MLE does not have explicit form, Chiba \cite{Chiba} proposed an M--estimator based on the likelihood function, and studied its asymptotic properties when the Hurst index $H$ is contained in $(\frac{1}{4},\frac{1}{2})$.

  The parametric inference for diffusion processes with small white noise has been well developed (see, e.g., Kutoyants \cite{s1Kutoyants}, \cite{s2Kutoyants}, Uchida and Yoshida \cite{Uchida}, Yoshida \cite{s1Yoshida} and \cite{s2Yoshida}). However, parametric estimation problems for the stochastic differential equation driven by small fractional Brownian motion has not been analyzed yet. The main tool to obtain the asymtotic properties of estimators when they do not have explicit expression is investigated by Ibragimov and Has'minskii \cite{Ibragimov}. In our case, the MLE does not have explicit form and we rely on the approach of Ibragimov and Has'minski. Their approach is based on the analysis of the likelihood ratio random field, where the large deviation inequality plays an important role to derive the asymtotic properties. We aim to deduce asymtotic properties of the maximum likelihood estimator when $\varepsilon\rightarrow 0$ in the spirit of Ibragimov and Has'minskii.

  This paper organized as follows: in Section 2 we make some notations and assumptions to state our main results. In Sect. 3 we prove main results. Most of the proof is checking out the sufficient conditions of the polynomial type large deviation inequality investigated by Yoshida \cite{Yoshida}.
 \section{Main results}

  Let $(\Omega,\mathcal{F},P)$ be a probability space. We assume that the parameter space $\Theta\in\mathbb{R}^d$ to be bounded, oepn and convex domain admitting Sobolev embedding $W^{1,p}(\Theta)\hookrightarrow C(\bar{\Theta})$ for $p>d$. We aim to estimate the unknown parameter $\theta_0\in\Theta$ in the equation \eqref{SDE} from completely observed data $\{X_t^\varepsilon\}_{t\in[0,T]}$. Let us define some functions appearing in the likelihood function for the equation \eqref{SDE}. We first recall the basic definitions of fractional calculus. Let $f\in L^1(a,b)$ for $a<b$ and $\alpha>0$.
  The fractional Riemann--Liouville integrals of $f$ of order $\alpha$ are defined for almost all $x\in(a,b)$ by
\begin{equation*}
  I_{a+}^{\alpha}f(x):=\frac{1}{\Gamma(\alpha)}\int_a^x(x-y)^{\alpha-1}f(y)dy,
\end{equation*}
and
\begin{equation*}
  I_{b-}^{\alpha}f(x):=\frac{1}{\Gamma(\alpha)}\int_x^b(y-x)^{\alpha-1}f(y)dy.
\end{equation*}
Let $I^\alpha_{a+}(L^p(a,b))$ (resp.$I^\alpha_{b-}(L^p(a,b))$) be the image of $L^p(a,b)$ by the operator $I^\alpha_{a+}(a,b)$  (resp.$I^\alpha_{b-}(L^p(a,b))$). If $f\in I^\alpha_{a+}(L^p(a,b))$ (resp.$I^\alpha_{b-}(L^p(a,b))$) and $0<\alpha<1$ then the Weyl derivative are defined by
\begin{equation*}
  D_{a+}^{\alpha}f(x):=\frac{1}{\Gamma(1-\alpha)}\left(\frac{f(x)}{(x-a)^{\alpha}}+\alpha\int_a^x\frac{f(x)-f(y)}{(x-y)^{\alpha+1}}dy\right),
\end{equation*}
and
\begin{equation*}
    D_{b-}^{\alpha}f(x):=\frac{1}{\Gamma(1-\alpha)}\left(\frac{f(x)}{(b-x)^{\alpha}}+\alpha\int_x^b\frac{f(x)-f(y)}{(x-y)^{\alpha+1}}dy\right).
\end{equation*}

 There are many well--known results for equation \eqref{SDE}. According to \cite{Nualart}, the existence and uniqueness of a strong solution to equation \eqref{SDE} follows under Assumption \ref{(A1)} and \ref{(A2)} described below. In addition, for every $0<\varepsilon<H$, the solution to \eqref{SDE} has $H-\varepsilon$ H\"older continuity. From the H\"older continuity of the solution to \eqref{SDE}, we can define the function (see Theorem 13.6 in \cite{Samko})
    \begin{equation*}
      Q_{H,\theta}^\varepsilon(t):=
      \begin{cases}
        \left(\varepsilon d_H\right)^{-1}t^{H-1/2}I_{0+}^{1/2-H}\left[(\cdot)^{1/2-H}b(X^\varepsilon_\cdot,\theta)\right]& {\rm if}~H<1/2\\
        \left(\varepsilon d_H\right)^{-1}t^{H-1/2}D_{0+}^{H-1/2}\left[(\cdot)^{1/2-H}b(X^\varepsilon_\cdot,\theta)\right]& {\rm if}~H>1/2.
      \end{cases}
    \end{equation*}
    where
    \begin{equation*}
      d_H:=\sqrt{\frac{2H\Gamma(\frac{3}{2}-H)\Gamma(H+\frac{1}{2})}{\Gamma(2-2H)}}.
    \end{equation*}
 For $0<s<t$, let $k_H^{-1}(t,s)$ be a functin given by
 \begin{equation*}
        k_H^{-1}(t,s):=
        \begin{cases}
            \frac{1}{d_H}s^{1/2-H}I_{t-}^{1/2-H}\left[(\cdot)^{H-1/2}\right]&{\rm if}~H<1/2\\
            \frac{1}{d_H}s^{1/2-H}D_{t-}^{H-1/2}\left[(\cdot)^{H-1/2}\right]&{\rm if}~H>1/2.
        \end{cases}
 \end{equation*}
We define a semimartingale $\{Z_t\}_{t\ge 0}$ as follows:
\begin{equation*}
 \begin{aligned}
   Z_t:&=\varepsilon^{-1}\int_0^Tk_H^{-1}(t,s)dX_s\\
   &=\int_0^tQ_{H,\theta}^\varepsilon(s)ds+W_t,
 \end{aligned}
\end{equation*}
where $\{W_t\}_{0\le t\le T}$ is a Wiener process. Note that we used the Volterra correspondence
\begin{equation*}
  W_t=\int_0^tk_H^{-1}(t,s)dB_s^H.
\end{equation*}
 Here we interpret the stochastic integral with respect to a fractional Brownian motion as a Wiener integral. The log--likelihood function $\mathbb{L}_{H,\varepsilon}$ for the equation \eqref{SDE} is given by
\begin{equation*}
 \mathbb{L}_{H,\varepsilon}(\theta):=\int_0^TQ_{H,\theta}^\varepsilon(t)dZ_t-\frac{1}{2}\int_0^TQ_{H,\theta}^\varepsilon(t)^2dt.
\end{equation*}
For more details about construction of the likelihood function, see \cite{Tudor}. We define the maximum lilelihood estimator by
\begin{equation*}
 \hat{\theta}_\varepsilon:=\argmax_{\theta\in\bar{\Theta}}\mathbb{L}_{H,\varepsilon}(\theta).
\end{equation*}
In order to state our main results, we make some notations. Let $\{x_t\}_{0\le t\le T}$ be the solution to the differential equation under the true value of the drift parameter:
\begin{equation}
 \label{ODE}
  \left \{
  \begin{aligned}
    \frac{dx_t}{dt}&=b(x_t,\theta_0)\\
    x_0&=x.
  \end{aligned}\right .
\end{equation}
We set the $d$--dimensional square matrix $\Gamma_H(\theta_0)$ as
\begin{equation*}
 \Gamma_H^{i,j}(\theta_0):=
 \begin{cases}
   c_1\int_0^Tt^{2H-1}\left(\int_0^ts^{1/2-H}(t-s)^{-1/2-H}\partial_{\theta_i}b(x_s,\theta_0)ds\right)\\
   ~~~~~~~~~~~~~~~\times\left(\int_0^ts^{1/2-H}(t-s)^{-1/2-H}\partial_{\theta_j}b(x_s,\theta_0)ds\right)dt& {\rm if}~H<1/2\\
   \int_0^T\Biggl(c_2t^{1/2-H}\partial_{\theta_i} b(x_t,\theta_0)+c_3t^{H-1/2}\int_0^t\frac{\partial_{\theta_i}b(x_t,\theta_0)-\partial_{\theta_i}b(x_s,\theta_0)}{(t-s)^{H+1/2}}s^{1/2-H}ds\Biggl)\\
   \times\Biggl(c_2t^{1/2-H}\partial_{\theta_j} b(x_t,\theta_0)+c_3t^{H-1/2}\int_0^t\frac{\partial_{\theta_j}b(x_t,\theta_0)-\partial_{\theta_j}b(x_s,\theta_0)}{(t-s)^{H+1/2}}s^{1/2-H}ds\Biggl)dt&{\rm if}~H>1/2,
 \end{cases}
\end{equation*}
where
\begin{equation*}
 \begin{aligned}
   c_1&=\left(d_H\Gamma(1/2-H)\right)^{-2}\\
   c_2&=\left(d_H\Gamma(3/2-H)\right)^{-1}\left\{1+(H-1/2)\int_0^1\frac{s^{1/2-H}-1}{(1-s)^{-H-1/2}}ds\right\}\\
   c_3&=(H-1/2)\left(d_H\Gamma(3/2-H)\right)^{-1}.
 \end{aligned}
\end{equation*}
Let
\begin{equation*}
  \begin{aligned}
      \mathbb{Y}_{H,\varepsilon}(\theta):=&\varepsilon^2\left(\mathbb{L}_{H,\varepsilon}(\theta)-\mathbb{L}_{H,\varepsilon}(\theta_0)\right)\\
      \mathbb{Y}_H(\theta):=&
      \begin{cases}
        c_1\int_0^Tt^{2H-1}\left\{\int_0^ts^{1/2-H}(t-s)^{-1/2-H}\left(b(x_s,\theta)-b(x_s,\theta_0)\right)ds\right\}^2dt& {\rm if}~H<1/2\\
        \int_0^T\Biggl(c_2t^{1/2-H} \left(b(x_t,\theta)-b(x_t,\theta_0)\right)\\
        \hspace{75pt}+c_3\int_0^t\frac{\left(b(x_t,\theta_0)-b(x_t,\theta)\right)-\left(b(x_s,\theta_0)-b(x_s,\theta)\right)}{(t-s)^{H+1/2}}s^{1/2-H}ds\Biggl)^2dt&{\rm if}~H>1/2.
      \end{cases}
  \end{aligned}
\end{equation*}
We introduce the following assumptions.
\begin{assumption}
  \label{(A1)}
  The function $b$ in \eqref{SDE} is of $C^{1,4}(\mathbb{R}\times\Theta;\mathbb{R})$--class such that for every $x\in\mathbb{R}$ and $\theta\in\Theta$,
the following growth conditions hold:
\begin{equation*}
  \begin{aligned}
    &|b(x,\theta)|\le c(1+|x|),~|\nabla^i_\theta b(x,\theta)|\le c(1+|x|^N),~|\nabla^i_\theta\partial_x b(x,\theta)|\le c(1+|x|^N),
  \end{aligned}
\end{equation*}
for $0\le i\le 4$ and some constants $c>0,~N\in\mathbb{N}$.
  \end{assumption}
  \begin{assumption}
    \label{(A2)}
    There exists $L>0$ such that for every $x,y\in\mathbb{R}$,
    \begin{equation*}
      \sup_{\theta\in\Theta}|b(x,\theta)-b(y,\theta)|\le L|x-y|.
    \end{equation*}
  \end{assumption}
  \begin{assumption}
    The matrix $\Gamma_H(\theta_0)$ is positive definite.
  \end{assumption}
  \begin{assumption}
    For every $\theta\in\Theta$, there exists a positive constant $\xi>0$ and $1<\rho\le 2$ such that
    \begin{equation*}
      -\mathbb{Y}_H(\theta)\le-\xi|\theta-\theta_0|^\rho
    \end{equation*}
  \end{assumption}
  The following theorem gives the asymtotic properties of the estimator $\hat{\theta}_\varepsilon$.
\begin{theo}
  \label{main}
 Suppose that the assumptions 1--4 are fulfilled. Then the estimator $\hat{\theta}_\varepsilon$ satisfies that
 \begin{equation*}
   \varepsilon^{-1}(\hat{\theta}_\varepsilon-\theta_0)\xrightarrow{d}N(0,\Gamma_H(\theta_0)^{-1})
 \end{equation*}
 as $\varepsilon\rightarrow 0$. Moreover, we have
 \begin{equation*}
   E\left[f\left(\varepsilon^{-1}(\hat{\theta}_\varepsilon-\theta_0)\right)\right]\rightarrow E[f(\xi)]
 \end{equation*}
 as $\varepsilon\rightarrow 0$ for every continuous function $f$ of polynomial
  growth, where $\xi\sim N(0,\Gamma_H(\theta_0)^{-1})$.
\end{theo}
 \section{Proofs}
 We first establish a lemma that is used frequently in this paper.
 \begin{lemm}
   \label{s}
   For every $s,t\in[0,T]$,
   \begin{equation*}
     \begin{aligned}
       |X_t^\varepsilon-x_t|&\le \varepsilon e^{LT}\sup_{0\le t\le T}|B_t^H|\\
       |X_t^\varepsilon|&\le e^{CT}\left(CT+\varepsilon\sup_{0\le t\le T}|B_t^H|\right),
     \end{aligned}
   \end{equation*}
   and
   \begin{equation*}
     |X_t^\varepsilon-X_s^\varepsilon|\le C\left(1+\sup_{0\le u\le T}|X_u|\right)|t-s|+\varepsilon|B_t^H-B_s^H|.
   \end{equation*}
 \end{lemm}
 \begin{proof}
     By \eqref{SDE} and \eqref{ODE}
     \begin{equation*}
       \begin{aligned}
                |X_t^\varepsilon-x_t|&\le\int_0^t\left|b(X_s^\varepsilon,\theta_0)-b(x_s,\theta_0)\right|+\varepsilon|B_t^H|\\
                &\le L\int_0^t|X_s^\varepsilon-x_s|ds+\varepsilon\sup_{0\le t\le T}|B_t^H|.
       \end{aligned}
     \end{equation*}
     By Gronwall's inequality, it follows that
     \begin{equation*}
       |X_t^\varepsilon-x_t|\le \varepsilon e^{Lt}\sup_{0\le t\le T}|B_t^H|,
     \end{equation*}
     and the first estimate follows. Other estimates hold true by the linear growth condition of the functon $b$.
 \end{proof}
To show the asymptotic property of estimator $\hat{\theta}_\varepsilon$, we apply the polynomial type large deviation inequality investigated by Yoshida (2011). Let $\mathbb{U}_\varepsilon(\theta_0):=\left\{u\in\mathbb{R}^d:\theta_0+\varepsilon u\in\Theta\right\}$ and define the random field $\mathbb{Z}_{H,\varepsilon}:\mathbb{U}_\varepsilon(\theta_0)\rightarrow\mathbb{R}_+$ by
\begin{equation*}
  \mathbb{Z}_{H,\varepsilon}(u)=\exp\left\{\mathbb{L}_{H,\varepsilon}(\theta_0+\varepsilon u)-\mathbb{L}_{H,\varepsilon}(\theta_0)\right\},~~~u\in \mathbb{U}_\varepsilon(\theta_0).
\end{equation*}
Applying Taylor's formula, we have
\begin{equation*}
  \log\mathbb{Z}_{H,\varepsilon}(u)=\varepsilon\nabla_\theta\mathbb{L}_{H,\varepsilon}(\theta_0)[u]-\frac{1}{2}u\Gamma_H(\theta_0)u^\top+R_\varepsilon(u),
\end{equation*}
where
\begin{equation*}
  \begin{aligned}
    R_\varepsilon(u)&=\frac{1}{2}u\left(\varepsilon^2\nabla_\theta^2\mathbb{L}_{H,\varepsilon}(\theta_0)-(-\Gamma_H(\theta_0))\right)u^\top\\
    &~~~~~~+\frac{1}{2}\varepsilon^3\int_0^1(1-s)^2\nabla_\theta^3\mathbb{L}_{H,\varepsilon}(\theta_0+s\varepsilon u)[u,u,u]ds.
  \end{aligned}
\end{equation*}
and $\nabla_\theta^3\mathbb{L}_{H,\varepsilon}(\theta)[u,v,w]=\sum_{i,j,k}\partial_i\partial_j\partial_k\mathbb{L}_{H,\varepsilon}(\theta)u_iv_jw_k$.
\begin{remark}
  \label{dif}
  The log--likelihood function $\mathbb{L}_{H,\varepsilon}$ is differentiable in $\theta$ under Assumption \ref{(A1)}, and we have
  \begin{equation*}
    \begin{aligned}
      \nabla_\theta\mathbb{L}_{H,\varepsilon}(\theta)&=\int_0^T\nabla_\theta Q_{H,\theta}^\varepsilon(t)dZ_t-\int_0^TQ_{H,\theta}^\varepsilon(t)\nabla_\theta Q_{H,\theta}^\varepsilon(t)dt,\\
      \nabla_\theta^2\mathbb{L}_{H,\varepsilon}(\theta_0)&=\int_0^T\nabla_\theta^2 Q_{H,\theta}^\varepsilon(t)dZ_t-\int_0^T\left(\nabla_\theta Q_{H,\theta}^\varepsilon(t)\right)^{\otimes2}dt-\int_0^TQ_{H,\theta}^\varepsilon(t)\nabla_\theta^2 Q_{H,\theta}^\varepsilon(t)dt.
    \end{aligned}
  \end{equation*}
\end{remark}
Throughout this paper, we will use the following notations:
\begin{notation}
For any $a,b\ge 0$, the symbol $a\lesssim b$ means that there exists a universal constant $C>0$ such that $a\le Cb$. When $C$ depends explicity on a specific quantity we shall indicate it explicity through the paper.
\end{notation}
The following lemma is one of the sufficient conditions for polynomial type large deviation inequality investigated by Yoshida [2011].
\begin{lemm}
  \label{e}
  For every $p\ge 2$,
  \begin{equation*}
    \sup_{0<\varepsilon<1}E\left[\left(\varepsilon^{-d_H}\left|\varepsilon^2\nabla_{\theta}^2\mathbb{L}_{H,\varepsilon}(\theta_0)-(-\Gamma_H(\theta_0))\right|\right)^p\right]<\infty,
  \end{equation*}
  where
  \begin{equation*}
    d_H=
    \begin{cases}
      1&{\rm if}~H<1/2\\
      1/2&{\rm if}~H>1/2.
    \end{cases}
  \end{equation*}
\end{lemm}
\begin{proof}
  By Remark \ref{dif}, it folows that
  \begin{equation*}
    \begin{aligned}
      \varepsilon^2\nabla_{\theta}^2\mathbb{L}_{H,\varepsilon}(\theta_0)-(-\Gamma_H(\theta_0))=\varepsilon^2\int_0^T\nabla_\theta^2 Q_{H,\theta_0}^\varepsilon(t)dW_t-\left(\varepsilon^2\int_0^T\left(\nabla_\theta Q_{H,\theta_0}^\varepsilon(t)\right)^{\otimes2}dt-\Gamma_H(\theta_0)\right).
    \end{aligned}
  \end{equation*}
  At first, we consider the case of $H<1/2$. Note that
  \begin{equation*}
    Q_{H,\theta_0}^\varepsilon(t)=c_1^2\varepsilon^{-1}t^{H-1/2}\int_0^ts^{1/2-H}(t-s)^{-1/2-H}b(X_s^\varepsilon,\theta_0)ds.
  \end{equation*}
  By Burkholder's and Minkowski's inequalities and Lemma \ref{s}, the stochastic integral part is estimated as
  \begin{equation*}
    \begin{aligned}
      &E\left(\varepsilon^2\left|\int_0^T\partial_{\theta_i}\partial_{\theta_j} Q_{H,\theta_0}^\varepsilon(t)dW_t\right|\right)^p\lesssim \left(\varepsilon^4\int_0^T\left \|\partial_{\theta_i}\partial_{\theta_j} Q_{H,\theta_0}^\varepsilon(t)\right\|_{L^p(\Omega)}^2dt\right)^{p/2}\\
      &\lesssim\varepsilon^p\left(\int_0^Tt^{2H-1}\left|\int_0^ts^{1/2-H}(t-s)^{-1/2-H}\left\|\partial_{\theta_i}\partial_{\theta_j}b(X_s^\varepsilon,\theta_0)\right\|_{L^p(\Omega)}ds\right|^2dt\right)^{p/2}\\
      &\lesssim\varepsilon^p\sup_{0\le s\le T}\left\|1+|X_s^\varepsilon|^N\right\|^p_{L^p(\Omega)}\left(\int_0^Tt^{1-2H}dt\right)^{p/2}\lesssim\varepsilon^p,
    \end{aligned}
  \end{equation*}
  for every $i,j=1,\cdots,d$. We shall estimate the second part. For every $i,j=1,\cdots,d$,
  \begin{equation*}
    \begin{aligned}
      &\varepsilon^2\int_0^T\partial_{\theta_i} Q_{H,\theta_0}^\varepsilon(t)\partial_{\theta_j} Q_{H,\theta_0}^\varepsilon(t)dt-\Gamma_H^{i,j}(\theta_0)\\
      &=c_1^2\int_0^Tt^{2H-1}\Biggl\{\left(\int_0^t s^{1/2-H}(t-s)^{-1/2-H}\partial_{\theta_i} b(X_s^\varepsilon,\theta_0)ds\right)\left(\int_0^t s^{1/2-H}(t-s)^{-1/2-H}\partial_{\theta_j} b(X_s^\varepsilon,\theta_0)ds\right)\\
      &~~~~~~~~~~~~~~~~~~~-\left(\int_0^t s^{1/2-H}(t-s)^{-1/2-H}\partial_{\theta_i} b(x_s,\theta_0)ds\right)\left(\int_0^t s^{1/2-H}(t-s)^{-1/2-H}\partial_{\theta_j} b(x_s,\theta_0)ds\right)\Biggl\}dt\\
      &=c_1^2\int_0^Tt^{2H-1}\Biggl\{\left(\int_0^t s^{1/2-H}(t-s)^{-1/2-H}\partial_{\theta_i}\left(b(X_s^\varepsilon,\theta_0)-b(x_s,\theta_0)\right)ds\right)\\
      &\hspace{100pt}\times\left(\int_0^t s^{1/2-H}(t-s)^{-1/2-H}\partial_{\theta_j}b(X_s^\varepsilon,\theta_0)ds\right)\\
      &\hspace{75pt}-\left(\int_0^t s^{1/2-H}(t-s)^{-1/2-H}\partial_{\theta_i} b(x_s,\theta_0)ds\right)\\
      &\hspace{100pt}\times\left(\int_0^t s^{1/2-H}(t-s)^{-1/2-H}\partial_{\theta_j}\left( b(x_s,\theta_0)-b(X_s^\varepsilon,\theta_0)\right)ds\right)\Biggl\}dt.
    \end{aligned}
  \end{equation*}
  By H\"older's ineq., Minkowski's ineq. and Lemma \ref{dif}, we can show that
  \begin{equation*}
    \begin{aligned}
      &E\Biggl|\int_0^Tt^{2H-1}\left(\int_0^t s^{1/2-H}(t-s)^{-1/2-H}\partial_{\theta_i}\left(b(X_s^\varepsilon,\theta_0)-b(x_s,\theta_0)\right)ds\right)\\
      &\hspace{50pt}\times\left(\int_0^t s^{1/2-H}(t-s)^{-1/2-H}\partial_{\theta_j}b(X_s^\varepsilon,\theta_0)ds\right)dt\Biggl|^p\\
      &\le\Biggl|\int_0^Tt^{2H-1}\left(\int_0^t s^{1/2-H}(t-s)^{-1/2-H}\left\|\partial_{\theta_i}\left(b(X_s^\varepsilon,\theta_0)-b(x_s,\theta_0)\right)\right\|_{L^{2p}(\Omega)}ds\right)\\
      &\hspace{50pt}\times\left(\int_0^t s^{1/2-H}(t-s)^{-1/2-H}\left\|\partial_{\theta_j}b(X_s^\varepsilon,\theta_0)\right\|_{L^{2p}(\Omega)}ds\right)dt\Biggl|^p\\
      &\lesssim\sup_{0\le s\le T}\left\|\left(1+|X_s^\varepsilon|^N+|x_s|^N\right)\left|X_s^\varepsilon-x_s\right|\right\|_{L^{2p}(\Omega)}^p\left(\int_0^Tt^{1-2H}dt\right)
      \lesssim\varepsilon^p.
    \end{aligned}
  \end{equation*}
Therefore
\begin{equation*}
  \sup_{0<\varepsilon<1}E\left[\left(\varepsilon^{-1}\left|\varepsilon^2\nabla_{\theta}^2\mathbb{L}_{H,\varepsilon}(\theta_0)-(-\Gamma_H(\theta_0))\right|\right)^p\right]<\infty.
\end{equation*}
The case where $H>1/2$. We have that
\begin{equation*}
  \begin{aligned}
    Q_{H,\theta_0}^\varepsilon(t)=c_1\varepsilon^{-1}t^{1/2-H}b(X_t^\varepsilon,\theta_0)+c_2\varepsilon^{-1}t^{H-1/2}\int_0^t\frac{b(X_t^\varepsilon,\theta_0)-b(X_s^\varepsilon,\theta_0)}{(t-s)^{H+1/2}}s^{1/2-H}ds.
  \end{aligned}
\end{equation*}
In a similar way with the case $H<1/2$, the stochastic integral part is evaluated as
\begin{equation*}
  \begin{aligned}
E\left(\varepsilon^2\left|\int_0^T\partial_{\theta_i}\partial_{\theta_j} Q_{H,\theta_0}^\varepsilon(t)dW_t\right|\right)^p&\lesssim \left(\varepsilon^4\int_0^T\left \|\partial_{\theta_i}\partial_{\theta_j} Q_{H,\theta_0}^\varepsilon(t)\right\|_{L^p(\Omega)}^2dt\right)^{p/2}\\
&\lesssim\varepsilon^p\Biggl\{\int_0^Tt^{1-2H}\|\partial_{\theta_i}\partial_{\theta_j}b(X_t^\varepsilon,\theta_0)\|_{L^p(\Omega)}^2dt\\
&\hspace{10pt}+\int_0^Tt^{2H-1}\left\|\int_0^t\frac{\partial_{\theta_i}\partial_{\theta_j}b(X_t^\varepsilon,\theta_0)-\partial_{\theta_i}\partial_{\theta_j}b(X_s^\varepsilon,\theta_0)}{(t-s)^{H+1/2}}s^{1/2-H}ds\right\|_{L^p(\Omega)}^2dt\Biggl\}^{p/2}\\
&\lesssim\varepsilon^p\Biggl\{\left\|1+\sup_{0\le t\le T}E|X_t^\varepsilon|^{2N}\right\|_{L^p(\Omega)}^2\int_0^Tt^{1-2H}dt\\
&\hspace{10pt}+\int_0^Tt^{2H-1}\left\|\int_0^t\frac{\left(1+|X_t^\varepsilon|^N+|X_s^\varepsilon|^N\right)|X_t^\varepsilon-X_s^\varepsilon|}{(t-s)^{H+1/2}}s^{1/2-H}ds\right\|_{L^p(\Omega)}^2dt\Biggl\}^{p/2}\\
&\lesssim\varepsilon^p\left\{1+\int_0^Tt^{2H-1}\left(\int_0^t\frac{\left\|X_t^\varepsilon-X_s^\varepsilon\right\|_{L^{2p}(\Omega)}}{(t-s)^{H+1/2}}s^{1/2-H}ds\right)^2dt\right\}^{p/2}\lesssim\varepsilon^p,
  \end{aligned}
\end{equation*}
for every $i,j=1,\cdots,d$. We estimate the term $\varepsilon^2\int_0^T\left(\nabla_\theta Q_{H,\theta_0}^\varepsilon(t)\right)^{\otimes2}dt-\Gamma_H(\theta_0)$. For every $i,j=1,\cdots,d$
\begin{equation*}
  \begin{aligned}
    &\varepsilon^2\int_0^T\partial_{\theta_i} Q_{H,\theta_0}^\varepsilon(t)\partial_{\theta_j} Q_{H,\theta_0}^\varepsilon(t)dt-\Gamma_H^{i,j}(\theta_0)\\
    &=c_2\left(\int_0^Tt^{1-2H}\left\{\partial_{\theta_i} b(X_t^\varepsilon,\theta_0)\partial_{\theta_j} b(X_t^\varepsilon,\theta_0)-\partial_{\theta_i} b(x_t,\theta_0)\partial_{\theta_j} b(x_t,\theta_0)\right\}dt\right)\\
    &+(c_2c_3)^{1/2}\Biggl(\int_0^T\partial_{\theta_i}b(X_t^\varepsilon,\theta_0)\int_0^t\frac{\partial_{\theta_j}b(X_t^\varepsilon,\theta_0)-\partial_{\theta_j}b(X_s^\varepsilon,\theta_0)}{(t-s)^{H+1/2}}s^{1/2-H}dsdt\\
    &\hspace{50pt}-\int_0^T\partial_{\theta_i}b(x_t,\theta_0)\int_0^t\frac{\partial_{\theta_j}b(x_t,\theta_0)-\partial_{\theta_j}b(x_s,\theta_0)}{(t-s)^{H+1/2}}s^{1/2-H}dsdt\Biggl)\\
    &+c_3\Biggl(\int_0^T\left(\int_0^t\frac{\partial_{\theta_i}b(X_t^\varepsilon,\theta_0)-\partial_{\theta_i}b(X_s^\varepsilon,\theta_0)}{(t-s)^{H+1/2}}s^{1/2-H}ds\right)\left(\int_0^t\frac{\partial_{\theta_j}b(X_t^\varepsilon,\theta_0)-\partial_{\theta_j}b(X_s^\varepsilon,\theta_0)}{(t-s)^{H+1/2}}s^{1/2-H}ds\right)dt\\
    &\hspace{30pt}-\int_0^T\left(\int_0^t\frac{\partial_{\theta_i}b(x_t,\theta_0)-\partial_{\theta_i}b(x_s,\theta_0)}{(t-s)^{H+1/2}}s^{1/2-H}ds\right)\left(\int_0^t\frac{\partial_{\theta_j}b(x_t,\theta_0)-\partial_{\theta_j}b(x_s,\theta_0)}{(t-s)^{H+1/2}}s^{1/2-H}ds\right)dt\Biggl).
  \end{aligned}
\end{equation*}
Using Lemma \ref{s}, we have
\begin{equation*}
  \begin{aligned}
    &E\left|\int_0^Tt^{1-2H}\left\{\partial_{\theta_i} b(X_t^\varepsilon,\theta_0)\partial_{\theta_j} b(X_t^\varepsilon,\theta_0)-\partial_{\theta_i} b(x_t,\theta_0)\partial_{\theta_j} b(x_t,\theta_0)\right\}dt\right|^p\\
    &\lesssim E\left|\int_0^Tt^{1-2H}\left\{\partial_{\theta_i} b(X_t^\varepsilon,\theta_0)\left(\partial_{\theta_j} b(X_t^\varepsilon,\theta_0)-\partial_{\theta_j} b(x_t,\theta_0)\right)+\partial_{\theta_j} b(x_t,\theta_0)\left(\partial_{\theta_i} b(X_t^\varepsilon,\theta_0)-\partial_{\theta_i} b(x_t,\theta_0)\right)\right\}\right|^p\\
    &\lesssim E\left|\int_0^Tt^{1-2H}(1+|X_t^\varepsilon|^N+|x_t|^N)^2\left|X_t^\varepsilon-x_t\right|\right|^p\lesssim\varepsilon^p.
  \end{aligned}
\end{equation*}
We shall estimate the second term. Note that
\begin{equation*}
  \begin{aligned}
    &\left\|\partial_{\theta_j}b(X_t^\varepsilon,\theta_0)-\partial_{\theta_j}b(X_s^\varepsilon,\theta_0)-\partial_{\theta_j}b(x_t,\theta_0)+\partial_{\theta_j}b(x_s,\theta_0)\right\|_{L^p(\Omega)}\\
    &\lesssim\left(\left\|\partial_{\theta_j}b(X_t^\varepsilon,\theta_0)-\partial_{\theta_j}b(X_s^\varepsilon,\theta_0)\right\|_{L^p(\Omega)}+\left\|\partial_{\theta_j}b(x_t,\theta_0)-\partial_{\theta_j}b(x_s,\theta_0)\right\|_{L^p(\Omega)}\right)^{1/2}\\
    &\times\left(\left\|\partial_{\theta_j}b(X_t^\varepsilon,\theta_0)-\partial_{\theta_j}b(x_s,\theta_0)\right\|_{L^p(\Omega)}+\left\|\partial_{\theta_j}b(X^\varepsilon_s,\theta_0)+\partial_{\theta_j}b(x_s,\theta_0)\right\|_{L^p(\Omega)}\right)^{1/2}\\
    &\lesssim\varepsilon^{1/2}|t-s|^{H/2}.
  \end{aligned}
\end{equation*}
 Thus we obtain that
\begin{equation*}
  \begin{aligned}
    &E\Biggl|\int_0^T\Biggl(\partial_{\theta_i}b(X_t^\varepsilon,\theta_0)\int_0^t\frac{\partial_{\theta_j}b(X_t^\varepsilon,\theta_0)-\partial_{\theta_j}b(X_s^\varepsilon,\theta_0)}{(t-s)^{H+1/2}}s^{1/2-H}ds\\
    &\hspace{50pt}-\partial_{\theta_i}b(x_t,\theta_0)\int_0^t\frac{\partial_{\theta_j}b(x_t,\theta_0)-\partial_{\theta_j}b(x_s,\theta_0)}{(t-s)^{H+1/2}}s^{1/2-H}ds\Biggl)dt\Biggl|^p\\
    &=E\Biggl|\int_0^T\Biggl(\left(\partial_{\theta_i}b(X_t^\varepsilon,\theta_0)-\partial_{\theta_i}b(x_t,\theta_0)\right)\int_0^t\frac{\partial_{\theta_j}b(X_t^\varepsilon,\theta_0)-\partial_{\theta_j}b(X_s^\varepsilon,\theta_0)}{(t-s)^{H+1/2}}s^{1/2-H}ds\\
    &\hspace{50pt}+\partial_{\theta_i}b(x_t,\theta_0)\int_0^t\frac{\partial_{\theta_j}b(X_t^\varepsilon,\theta_0)-\partial_{\theta_j}b(X_s^\varepsilon,\theta_0)-\partial_{\theta_j}b(x_t,\theta_0)+\partial_{\theta_j}b(x_s,\theta_0)}{(t-s)^{H+1/2}}s^{1/2-H}ds\Biggl)dt\Biggl|^p\\
    &\lesssim\left(\int_0^T\left\|(1+|X_t^\varepsilon|^N+|x_t|^N)|X_t^\varepsilon-x_t|\right\|_{L^{2p}(\Omega)}\int_0^t\frac{\left\|\left(1+|X_t^\varepsilon|^N+|X_s^\varepsilon|^N\right)|X_t^\varepsilon-X_s^\varepsilon|\right\|_{L^{2p}(\Omega)}}{(t-s)^{H+1/2}}s^{1/2-H}ds\right)^p\\
    &+\left(\int_0^T(1+|x_t|^N)\int_0^t\frac{\left\|\partial_{\theta_j}b(X_t^\varepsilon,\theta_0)-\partial_{\theta_j}b(X_s^\varepsilon,\theta_0)-\partial_{\theta_j}b(x_t,\theta_0)+\partial_{\theta_j}b(x_s,\theta_0)\right\|_{L^p(\Omega)}}{(t-s)^{H+1/2}}s^{1/2-H}dsdt\right)^p\\
    &\lesssim\varepsilon^p+\varepsilon^{p/2}.
  \end{aligned}
\end{equation*}
We can estimate the third term in a similar way the second term and we complete the proof.
\end{proof}
\begin{lemm}
  \label{3}
    For every $p\ge 2$,
  \begin{equation*}
    \sup_{0<\varepsilon<1}E\left[\sup_{\theta\in\Theta}\left|\varepsilon^2\nabla_{\theta}^3\mathbb{L}_{H,\varepsilon}(\theta)\right|^p\right]<\infty.
  \end{equation*}
\end{lemm}
\begin{proof}
  By Sobolev's inequality, for every $p>d$
  \begin{equation*}
    \sup_{\theta\in\Theta}\left|\nabla_{\theta}^3\mathbb{L}_{H,\varepsilon}(\theta_0)\right|^p\lesssim\int_\Theta\left(\left|\nabla_{\theta}^3\mathbb{L}_{H,\varepsilon}(\theta)\right|^p+\left|\nabla_{\theta}^4\mathbb{L}_{H,\varepsilon}(\theta)\right|^p\right)d\theta.
  \end{equation*}
  By the same argument as in the proof of Lemma \ref{e}, we can show that
  \begin{equation*}
    \sup_{0<\varepsilon<1}\varepsilon^{2p}E\left[\int_\Theta\left(\left|\nabla_{\theta}^3\mathbb{L}_{H,\varepsilon}(\theta)\right|^p+\left|\nabla_{\theta}^4\mathbb{L}_{H,\varepsilon}(\theta)\right|^p\right)d\theta\right]<\infty.
  \end{equation*}
\end{proof}
We can check the following lemma holds true. The proof is similar to the one for Lemmas \ref{e} and \ref{3}.
\begin{lemm}
  \label{4}
  For every $p\ge 2$,
  \begin{equation*}
    \sup_{0<\varepsilon<1}E|\varepsilon\partial_\theta\mathbb{L}_{H,\varepsilon}(\theta_0)|^p<\infty,
  \end{equation*}
  and
  \begin{equation*}
    \sup_{0<\varepsilon<1}E\left(\sup_{\theta\in\Theta}\varepsilon\left|\mathbb{Y}_{H,\varepsilon}(\theta)-\mathbb{Y}_H(\theta)\right|\right)^p<\infty
  \end{equation*}
\end{lemm}
From Lemma \ref{4} and the proof of Lemma \ref{e}, we obtain that
\begin{equation}
\label{n}
\varepsilon\partial_\theta\mathbb{L}_{H,\varepsilon}(\theta_0)\xrightarrow{d}N(0,\Gamma_H(\theta_0)),
\end{equation}
as $\varepsilon\rightarrow 0$ by the martingale central limit theorem. Moreover, Lemmas \ref{e}, \ref{3} and the convergence \eqref{n} give the local asymptotic normality of $Z_{H,\varepsilon}(u)$:
\begin{equation*}
  Z_{H,\epsilon}(u)\xrightarrow{d}Z_{H}(u):=\exp\left(\Delta_H(\theta_0)u-\frac{1}{2}\Gamma_H(\theta_0)[u,u]\right),
\end{equation*}
where $\Delta_H(\theta_0)\sim N(0,\Gamma_H(\theta_0))$. Now Theorem 3 of Yoshida (2011) yields the inequality
\begin{equation}
  \sup_{0<\varepsilon<1}P\left[\sup_{|u|\ge r}\mathbb{Z}_{H,\varepsilon}(u)\ge e^{-r}\right]\lesssim r^{-L}
\end{equation}
hold for any $r>0$ and $L>0$. Since $u_\varepsilon:=\varepsilon^{-1}(\theta_\varepsilon-\theta_0)$ maximizes the random field $Z_{H,\varepsilon}$, the sequence $\{f(u_\varepsilon)\}_{\varepsilon}$ is uniformly integrable for every continuous function $f$ such that for every $x\in\mathbb{R}$, $f(x)\lesssim 1+|x|^N$ for some $N>0$. Indeed,
\begin{equation*}
  \begin{aligned}
    \sup_{0<\varepsilon<1}P\left(|u_\varepsilon|\ge r\right)&\le\sup_{0<\varepsilon<1}P\left(\sup_{|u|\ge r}\mathbb{Z}_{H,\varepsilon}(u)\ge\mathbb{Z}_{H,\varepsilon}(0)\right)
    \lesssim r^{-L},
  \end{aligned}
\end{equation*}
for every $r>0$ and $L>0$. Thus
\begin{equation*}
  \begin{aligned}
     \sup_{0<\varepsilon<1}E[|f(u_\varepsilon)|]&\lesssim 1+\int_0^\infty\sup_{0<\varepsilon<1} P\left(|u_\varepsilon|>r^{1/N}\right)dr<\infty.
  \end{aligned}
\end{equation*}
  Let $B(R):=\left\{u\in\mathbb{R}^d; |u|\le R\right\}$. In the sequel, we prove that
\begin{equation}
  \label{f}
  \log\mathbb{Z_{H,\varepsilon}}\xrightarrow{d}\log\mathbb{Z}_{H,\varepsilon}~~~{\rm in}~C(B(R)),
\end{equation}
as $\varepsilon\rightarrow 0$. If we can show the convergence \eqref{f}, we obtain the asymptotic normality:
\begin{equation*}
   \varepsilon^{-1}(\theta_\varepsilon-\theta_0)\xrightarrow{d}N(0,\Gamma_H(\theta_0)^{-1})
\end{equation*}
as $\varepsilon\rightarrow 0$ by Theorem 5 in Yoshida \cite{Yoshida}. Due to linearrity in $u$ of the weak convergence term $\varepsilon\nabla_\theta\mathbb{L}_{H,\varepsilon}(\theta_0)[u]$, the convergence of finite--dimensional distribution holds true. It remains to show the tightness of the family $\left\{\log\mathbb{Z}_{H,\varepsilon}(u)\right\}_{u\in B(R)}$. By the Kolmogorov tightness criterion, it suffices to show that
for every $R>0$ there exists a constant $p>0,~\gamma>d$ and $C>0$ such that
\begin{equation}
  \label{tight}
  E\left|\log\mathbb{Z}_{H,\varepsilon}(u_1)-\log\mathbb{Z}_{H,\varepsilon}(u_2)\right|^p\le C|u_1-u_2|^\gamma,
\end{equation}
for $u_1,u_2\in B(R).$ For a number $p>0$ large enough, the inequality \eqref{tight} is shown easily by Lemmas \ref{e}, \ref{3} and \ref{4}. Therefore, we complete the proof.
\paragraph{Acknowledgments}
The second author was partially supported by JSPS KAKENHI Grant Numbers JP21K03358 and JST CREST JPMJCR14D7, Japan.

\end{document}